\documentclass[11pt,letterpaper]{amsart}
\usepackage{srcltx}
\usepackage{amsmath}
\usepackage{amsfonts}
\usepackage{mathrsfs}
\usepackage[all]{xy}
\usepackage{amsthm}

\usepackage{latexsym}
\usepackage{amssymb}
\usepackage{graphicx}
\usepackage{verbatim}
\usepackage{fancyhdr}

\numberwithin{equation}{section}

\newtheorem{thm}[equation]{Theorem}
\newtheorem{lemma}[equation]{Lemma}
\newtheorem{prop}[equation]{Proposition}
\newtheorem{cor}[equation]{Corollary}
\newtheorem{rmk}[equation]{Remark}

\newcommand{\Z}[0]{\mathbb{Z}}

\newcommand{\N}[0]{\mathbb{N}}
\newcommand{\p}[0]{\mathbb{P}}

\newcommand{\F}[0]{\mathbb{F}}

\newcommand{\LL}[0]{\mathscr{L}}
\newcommand{\OO}[0]{\mathcal{O}}
\newcommand{\RR}[0]{\mathscr{R}}

\begin{document}

\title{A new proof of the Alexander-Hirschowitz interpolation theorem}
\author{Elisa Postinghel}

 \address{Dipartimento di Matematica, Universit\`a Roma Tre \\L.go S. L. Murialdo 1,
00146- Roma, Italy}
\email{ postingh@mat.uniroma3.it}

\maketitle

\begin{abstract} 
The classical polynomial interpolation problem in several variables can be generalized to the case of points with greater multiplicities. What is known, as yet, is essentially concentrated in the Alexander-Hirschowitz Theorem which says that a general collection of double points in $\mathbb{P}^r$ gives independent conditions on the linear system $\mathcal{L}$ of the hypersurfaces of degree $d$, with a well known list of exceptions. We  present a new proof of this theorem which consists in performing degenerations of $\p^r$ and analyzing how $\mathcal{L}$ degenerates.
 \end{abstract}


\section*{Introduction}

Fix $p_1,\dots,p_n \in \p^r$ distinct points  and fix $m_1,\dots,m_n$ positive integers. Define $\LL_{r,d}$ to be the linear system of hypersurfaces of $\p^r$ of degree $d$ and consider 
$$
\LL:=\LL_{r,d}(m_1,\dots,m_n)
$$
 the sub-linear system  of those divisors of $\LL_{r,d}$ having multiplicity at least $m_i$ at $p_i$, $i=1,\dots,n$. 
 Its \emph{virtual dimension} is defined to be
$$
v(\LL):={{r+d}\choose r}-1-\sum_{i=1}^{n}{{r+m_i-1}\choose r},
$$
i.e. the dimension of $\LL_{r,d}$ minus the number of conditions imposed by the multiple points. The actual dimesion of $\LL$ cannot be less than $-1$, hence we define the \emph{expected dimension} to be
$$
e(\LL):=\max \{v(\LL), -1\}.
$$
If the conditions imposed by the assigned points are not linearly independent, the actual dimension of $\LL$ is greater that the expected one: in that case we say that $\LL$ is \emph{special}. Otherwise, if the actual and the expected dimension coincide, we say that $\LL$ is \emph{non-special}.\\
The \emph{dimensionality problem} consists in investigating, given a linear system $\LL$, if it is non-special. The dimension of $\LL$ is \emph{upper-semicontinuous} in the position of the points in $\p^r$; it achieves its minimum value when they are in \emph{general position}. 
Let $Z$ be a scheme of lenght $\sum_{i=1}^n {r+m_i-1 \choose r}$ given by $n$ fat points in general position and consider the following restriction exact sequence
$$
0\rightarrow \LL=\LL_{r,d}(m_1,\dots,m_n)\rightarrow \LL_{r,d}\rightarrow{\LL_{r,d}}_{|Z}.
$$
In cohomology we get
$$
0\rightarrow H^0(\p^r,\LL)\rightarrow H^0(\p^r,\LL_{r,d})\rightarrow  H^0(Z,{\LL_{r,d}}_{|Z})\rightarrow H^1(\p^r,\LL)\rightarrow 0,
$$
being $h^1(\p^r,\LL_{r,d})=0$. Thus $\LL$ is non-special if and only if 
$
h^0(\p^r,\LL)\cdot h^1(\p^r,\LL) = 0.
$

For $r=1$ and general points, the system $\LL_{1,d}(m_1,\dots,n_n)$ is always non-special. Furthermore, if all points have multiplicity one the system $\LL_{r,d}(1^n)$ is also non-special.
 However, the problem becomes more and more complicated in several variables and higher multiplicities.
What is known is  Theroem \ref{A-H}, a result due to J. Alexander and A. Hirschowitz. They classify the special cases for $r\geq2$ and $m_1=\cdots=m_n=2$. 
\begin{thm}[Alexander-Hirschowitz]\label{A-H}
The linear system $\LL_{r,d}(2^n)$ is non-special except in the following cases:
$$
\begin{tabular}{c|cccccc}
r  &  $\forall$ & 2 &  3 &  4 &  4 \\
\hline 
d & 2 & 4 & 4 & 4 & 3 \\
\hline 
n & $\leq$ r & 5 & 9 & 14 & 7
\end{tabular}
$$
\end{thm}

A natural approach to the dimensionality problem of linear systems is via \emph{degenerations}. Degenerations allow  to move the multiple base points of the linear system in special position, arguing with a semicontinuity argument. More precisely, if one finds a specialization of the points, which is good in the sense that  the corresponding limit linear system is non-special, then also the original one is non-special. Computing the limit linear system is in general delicate. Hirchowitz in \cite{H} elaborated a degeneration technique, which he called \emph{la m\'ethode d'Horace}, consisting in making iterated specializations of as many points as convenient on a fixed hyperplane and then applying induction on  dimension and  degree. To be more explicit, let $\LL:=\LL_{r,d}(2^n)$ be the linear system of hypersurfaces of $\p^r$ of degree $d$ singular at a collection of $n$ general points; the main idea of  Hirschowitz was to degenerate in such a way that $h$ of the $n$ points have support on a fixed hyperplane $\pi\subseteq\p^r$; one gets the so called \emph{Castelnuovo exact sequence}:
$$
0 \rightarrow \LL_{r,d-1}(2^{n-h},1^h)\rightarrow\LL\rightarrow\LL_{r-1,d}(2^h),
$$
where the $h$ base points of the kernel system are the residual of the $h$ double points specialized on $\pi$.
Thus, arguing by induction, if the two external systems are non-special with virtual dimension at least $-1$, which means that one does not lose any condition in this restriction procedure, i.e. $h^1(\p^r,\LL_{r,d-1}(2^{n-h},1^h))=h^1(\pi,\LL_{r-1,d}(2^{n}))=0$, then the system $\LL$ is non-special too. Unfortunately, this method does not cover all possible situations. A refined version, the so called \emph{m\'ethode d'Horace diff\'erentielle}, gives a general solution exploiting subsequent specializations of part of the double base points of the linear system to a hyperplane $\pi$. The original proof, of about a hundred pages proposed by  Alexander and  Hirschowitz, is contained in \cite{A}-\cite{AH1}-\cite{AH2}-\cite{AH3} and  simplified in \cite{AH4}. 

In 2002, K. Chandler presented an easier proof of Theorem \ref{A-H} for $d\geq4$ in \cite{Ch1}. She  proposes a simplified version of the Horace's method using the \emph{Curvilinear Lemma} (Lemma 4).  In the case with degree three, the method does not work because specializing to hyperplanes one must deal with quadrics which give riso to special systems. Another problem with cubics is that each of the lines joining pairs of points lies in the base locus of the linear system, hence the standard approach can fail because these lines meet $\pi$. K. Chandler transformed the obstruction caused by the presence of lines in the base locus  in an advantage and completed the proof of Theorem \ref{A-H}, see \cite{Ch2}. The innovation was to specialize some of the points onto a subspace $L$ of codimension $2$ and pairs of points on hyperplanes containing $L$.

A recent improvement of this argument is due to M. C. Brambilla and G. Ottaviani. In a beautiful paper (\cite{BO}) they offer a shorter proof of Theorem \ref{A-H} in the case  $d\geq 4$ and propose a new and simpler degeneration argument in the cubic case. Their  argument is similar to that of  Chandler, but it is more effective. Their main idea is to choose a subspace $L$ of codimension three, instead of two, on which they specialize the points. This choice really simplifies the arithmetic side of the problem.

C. Ciliberto and R. Miranda in \cite{CM1} and \cite{CM2} used a different degeneration construction, originally proposed by Z. Ran (\cite{R}) to study higher multiplicity interpolation problem, in particular to prove Theorem \ref{A-H} in the planar case. This approach consists in degenerating the plane to a reducible surface, with two components intersecting along a line, and simultaneously degenerating the linear system $\LL=\LL_{2,d}(2^n)$ to a linear system $\LL_0$ obtained as fibered product of  linear systems on the two components over the restricted system on their intersection. The limit linear system $\LL_0$ is somewhat easier than the original one, in particular this degeneration argument allows to use induction either on the degree or on the number of imposed multiple points. This contruction provides a recursive formula for the dimension of $\LL_0$ involving the dimensions of the systems on the two components. 

In this paper we generalize this approach to the case with $r\geq3$ and we complete the proof of Theorem \ref{A-H} with this method, exploiting induction on both $d$ and $r$. In Section \ref{Ideg} we describe our construction: it consists in blowing up a point $p\in\p^r$ and twisting by an appropriate negative multiple of the exceptional divisor, obtaining a reducible central fiber which is the union of the exceptional divisor   and of the strict transform  of the blowing up of $\p^r$ at $p$ in the central fiber of a trivial family $\p^r\times\Delta$ over a disc $\Delta$ with a line bundle which restricts to $\mathcal{O}_{\p^r}(d)$ on any fiber. The two components intersect along a $(r-1)$-dimensional variety  that is isomorphic to $\p^{r-1}$. Then we consider $n$ general points on the general fiber, we specialize them   on the two components of the central fiber and we study the corresponding limit linear systems. This argument does not suffice to cover all the cases, because of an arithmetic obstruction similar to the one that  Brambilla and Ottaviani met. Our idea is to perform further degenerations in order to handle these cases; the interested reader can find the details  in Section \ref{IIdeg}.

A tricky point of this approach is the study of the transversality of the restrictions of the systems on the intersection of the two components. In the planar case,  Ciliberto and Miranda proved it using the finitness of the set of inflection points of linear systems on $\p^1$ (\cite{CM1}, Proposition 3.1). In higher dimension transversality is  more complicated. In Section \ref{transIdeg} and in Section \ref{transIIdeg} we present our approach to this problem: if at least one of the two restricted systems is a complete linear system, then the dimension of the intersection is easily computed.
 Anyhow, this is not sufficient to finish  the proof of Theorem \ref{A-H}. For istance, it does not work in the cubic case. The solution to this obstacle is to blow up a codimension three subspace $L$ of $\p^r$, instead of a point. This approach to the cubic case is not so different from the one of  Brambilla and  Ottaviani; we propose it in Section \ref{cubics}.  
 
Also the quartic case must be analysed separately. Indeed, twisting by a negative multiple of the exceptional component of the central fiber, we reduce to quadrics that are special.  We show Theorem \ref{A-H} for quartics in Section \ref{quartics} by induction on $r$, with a geometric argument that exploits the property of cubics of containing all lines through two distinct  double points.

Our construction besides its intrinsic intent (on the way we prove non-specility of some interesting systems, see Section \ref{LF})
gives hope for further extensions to greater multiplicities.

\section{The special cases }\label{specials}
In this section we briefly describe the special cases of Theroem \ref{A-H}.
The linear system $\LL=\LL_{r,2}(2^{n})$, with $2\leq n\leq r$ consists of quadric cones with vertex containing the double $(n-1)$-dimensional linear subspace of $\p^r$ determined by the $n$ points: hence
$\textrm{dim}(\LL)=  {{r-n+2}\choose 2}-1 >e({\LL}).$
For $n\geq r+1$, the system $\LL_{r,2}(2^n)$ is empty.

Let $n={{r+2}\choose2}-1$, for $r=2,3,4$. The linear system $\LL_{r,4}(2^{n})$ is expected to be empty. Nevertheless it is special because there exists a (unique) quartic singular at the given points, i.e. the double quadric through them.

Through a general collection of seven points in $\p^4$ there exists a rational normal curve of degree four;  
its secant variety is a cubic hypersurface which is singular along the whole curve and in particular at the seven points. Thus $\LL_{4,3}(2^7)$ is special, having virtual dimension equal to $-1$.


\section{The first degeneration}\label{Ideg}

Let us first define the integers
$$
n^-=n^-(r,d):=\left\lfloor \frac{1}{r+1}{{r+4}\choose4}\right\rfloor, \ 
n^+=n^+(r,d):=\left\lceil \frac{1}{r+1}{{r+4}\choose4}\right\rceil.
$$
 
If non-speciality holds for a collection of $n^-$ double points, then it holds for a smaller number of double points. On the other hand, if there are no hypersurfaces of degree $d$ with $n^+$ general nodes, the same is true adding other nodes. 
Our aim is to prove   by induction on $r$ and $d$ that $\LL_{r,d}(2^n)$ is non-special for $n^-\leq n\leq n^+$, except in the list of Theorem \ref{A-H}.\\
The  technique consists in degenerating $\p^r$ to a reducible variety and  studying how a linear system on the general fiber degenerates. The limiting system will be easier than the general one, and this will enable us to use induction.

Let $\Delta$ be a complex disc with center at the origin. Consider the product $\mathcal{V}=\p^r \times \Delta$ with the natural projections $p_1$ and $p_2$. Let $V_t=\p^r \times \{t\}$ be the fiber of $p_2$ over $t\in\Delta$. Take a point $(p,{0})$ in the central fiber $V_0$ and blow it up to obtain a new $(r+1)$-fold $\mathcal{X}$ with the maps 
  $f:\mathcal{X} \rightarrow \mathcal{V}$,
  $\pi_1=p_1 \circ f$ 
 and $\pi_2=p_2 \circ f$:

\begin{displaymath}
\xymatrix{ 
\mathcal{X}  \ar[rr]_f \ar[drr]_{\pi_2} \ar@/^/[rrrr]^{\pi_1}  && \mathcal{V} \ar[rr]_{p_1} \ar[d]^{p_2}   && \p^r \\
&& \Delta }
\end{displaymath}
The so obtained flat morphism $\pi_2:\mathcal{X} \rightarrow \Delta$, with fiber $X_t=\pi_2^{-1}(t)$, $t \in\Delta$, produces a $1$-dimensional degeneration of $\p^r$. If $t \neq 0$ then $X_t=V_t$ is a $\p^r$, while for $t=0$ the fiber $X_0$ is the union of the strict transform $\F$ of $V_0$ and the exceptional divisor $\p\cong\p^r$ of the blow-up. The two varieties $\p$ and $\F$ meet transversally along a $(r-1)$-dimensional variety $R$ which is isomorphic to $\p^{r-1}$: it represents a hyperplane on $\p$ and the exceptional divisor on $\F$.\\
A line bundle  on $X_0$ corresponds to two line bundles,
 respectively on $\p$ and on $\F$, which agree on the intersection $R$. Precisely $\textrm{Pic}(X_0)=\textrm{Pic}(\p)\times_{\textrm{Pic}(R)}\textrm{Pic}(\F)$, where the Picard group of $\p$ is generated by $\mathcal{O}(1)$, while the Picard group of $\F$ is generated by the hyperplane class $H$ and the class $E$ of the exceptional divisor.\\
Consider the line bundle $\mathcal{O}_{\mathcal{X}}(d)=\pi^{\ast}_1(\mathcal{O}_{\p^r}(d))$:  its restriction to the general fiber $X_t\cong\p^r$  is isomorphic to $\mathcal{O}_{\p^r}(d)$, while on the central fiber the restrictions to $\p$ and $\F$ are $\mathcal{O}_{\p}$ and $\mathcal{O}_{\F}(dH)$ respectively.
Now let us execute a twist by the bundle $\mathcal{O}_{\mathcal{X}}(-(d-1)\p)$: the restriction to $X_t$ is still the same, while the restrictions to $\p$ and $\F$ become
$$
\mathcal{O}_{\p}(d-1) \textrm{ and } \mathcal{O}_{\F}(dH-(d-1)E):
$$ 
the resulting line bundle on $X_0$ is a flat limit of the bundle $\mathcal{O}_{\p^r}(d)$ on the general fiber. Such a limit is not unique.\\
We now consider the homogeneus linear system $\LL_t:=\LL=\LL_{r,d}(2^n)$ of the hypersurfaces of $\p^r$ of degree $d$ with $n$ assigned general points $p_{1,t},\dots,p_{n,t}$ of equal multiplicity $m=2$.
Fix a non-negative integer $b \leq n$ and specialize $b$ points generically on $\F$ and the other $n-b$ points generically on $\p$: i.e. take a flat family $\{p_{1,t}\dots,p_{n,t}\}_{t\in\Delta}$ such that $p_{1,0},\dots,p_{b,0}\in\F$ and $p_{b+1,0},\dots,p_{n,0}\in\p$. 
The limiting linear system $\LL_0$ on $X_0$  is formed by the divisors in the flat limit of the bundle $\mathcal{O}_{\p^r}(d)$ on the general fiber $X_t$, singular at $p_{1,0},\dots,p_{n,0}$.  This system restricts to $\F$ and to $\p$ to the following systems:
$$
\LL_{\p}= \LL_{r,d-1}(2^{n-b}) \textrm{ and }
\LL_{\F} = \LL_{r,d}(d-1,2^b),
$$
where the point of multiplicity $d-1$ is the point $p\in V_0\cong\p^r$ which we  blew up to obtain $\F$.
We say that the limit linear system $\LL_0$ is obtained from $\LL$ by a $(1,b)$-\emph{degeneration}, according to \cite{CM1}.
An element of $\LL_0$ consists either of a divisor on $\p$ and a divisor  on $\F$, both satisfying the conditions imposed by the multiple points, which restrict to the same divisor on $R$, or it is a divisor corresponding to a section of the bundle which is identically zero on $\p$ (or on $\F$) and which gives a general divisor in $\LL_{\F}$ (or in $\LL_{\p}$ respectively)  containing $R$ as a component.\\
If we denote by $l_0$ the dimension of $\LL_0$ on $X_0$, we have, by upper semicontinuity, that
$
l_0 \geq \dim(\LL_t)=\dim(\LL_{r,d}(2^n)).
$
\begin{lemma}
In the above notation, if $l_0 = e(\LL_{r,d}(2^n))$, then the linear system $\LL$ has the expected dimension, i.e. it is non-special.
\end{lemma}

Let us consider the restriction exact sequences to $R\cong \p^{r-1}\subset\p^r$:
$$
0 \rightarrow \hat{\LL}_{\p}\rightarrow \LL_{\p}\rightarrow \RR_{\p} \subseteq |\mathcal{O}_{\p^{r-1}}(d-1)|
\textrm{ and }0 \rightarrow \hat{\LL}_{\F}\rightarrow \LL_{\F}\rightarrow \RR_{\F} \subseteq |\mathcal{O}_{\p^{r-1}}(d-1)|,
$$
  where $\RR_{\p}$, $\RR_{\F}$ denote the restrictions of the systems  $\LL_{\p}$, $\LL_{\F}$ to $R$ and $\hat{\LL}_{\p}$, $\hat{\LL}_{\F}$ denote the kernel systems:
$$
\hat{\LL}_{\p}  = \LL_{r,d-2}(2^{n-b}) \textrm{ and }
\hat{\LL}_{\F}  =  \LL_{r,d}(d,2^b).
$$
The kernel $\hat{\LL}_{\p}$ consists of those sections of $\LL_{\p}$ which vanish identically on $R$, i.e. the divisors in $\LL_{\p}$ containing  $R\cong\p^{r-1}$ as component; the same holds for $\hat{\LL}_{\F}$.\\
We denote by $v_{\p}$, $v_{\F}$, $\hat{v}_{\p}$, $\hat{v}_{\F}$ and by $l_{\p}$, $l_{\F}$, $\hat{l}_{\p}$, $\hat{l}_{\F}$ the virtual and the actual dimensions of the various linear systems.
Our aim is to to compute  $l_0$ by recursion. The simplest cases occurs when all the divisors in $\LL_0$ come from a section which is identically zero on one of the two components: in those cases the matching sections of the other system must lie in the kernel of the restriction map. 
If, on the contrary, the divisors on $\LL_0$ consist of a divisor on $\p$ and a
divisor on $\F$, both not identically zero, which match on $R$, then the
dimension of $\LL_0$ depends on the dimension of the intersection
$\RR:=\RR_\p\cap\RR_\F$ of the restricted systems. 
A section of $H^0(X_0,\LL_0)$ is obtained by taking an element in $H^0(R,\RR)$ and choosing preimages of such an element: $h^0(X_0,\LL_0)=h^0(R,\RR)+h^0(\p,\hat{\LL}_\p)+h^0(\F,\hat{\LL}_\F)$. Thus, at the linear system level
\begin{eqnarray}\label{formula l_0} 
l_0=\textrm{dim}(\RR)+\hat{l}_{\p}+\hat{l}_{\F}+2. 
\end{eqnarray}
The crucial point is to compute the dimension of $\RR$, from which one obtains $l_0$. If the systems $\RR_{\p},\RR_{\F} \subset |\OO_{\p^{r-1}}(d-1)|$ are \emph{transversal}, i.e. if they intersect properly inside $|\mathcal{O}_{\p^{r-1}}(d-1)|$, then
$
\dim(\RR)=\max\left\{r_{\p}+r_{\F}-{{d+r-2}\choose{r-1}}+1,-1\right\} .
$

Notice that transversality holds if at least one between $\LL_\p$ and $\LL_\F$ cuts the complete series on $R$.


\subsection{Some useful lemmas}\label{LF}

For what  concerns the analysis of the linear system on $\p$ and the relative kernel system, we can exploit induction on $d$ because they are linear systems of hypersurfaces of lower degree with nodes. Actually this is the reason for  performing degenerations as described above. However, in general the systems $\LL_{\F}$ and $\hat{\LL}_{\F}$ are unknown because of the presence of a point of greater multiplicity in their base locus. This section is devoted to the study of such linear systems.
\begin{lemma}
\label{system capLF}
The linear system $\LL_{r,d}(d,2^b)$  is either special of dimension  $\dim(\LL_{r-1,d}(2^b))$, or it is empty. 
\end{lemma}
\begin{proof}
A hypersurface of $\p^r$ of degree $d$ having multiplicity $d$ at a point $p$ is a cone of degree $d$ with vertex at $p$. 
Let $p_1,\dots,p_b$ be the general double points. A divisor in $\LL_{r,d}(d,2^b)$ is a cone, with vertex at $p$, over a hypersurface of degree $d$ in a general hyperplane $\pi\cong\p^{r-1}$ of $\p^r$ ($p\notin\pi$) that must be singular at the $b$ points obtained from $p_1,\dots,p_b$ as projection from $p$ to $\pi$, which are general in $\pi$. 
$h^0(\p^r,\LL_{r,d}(d,2^b))=h^0(\p^{r-1},\LL_{r-1,d}(2^b))$; moreover $h^1(\p^r,\LL_{r,d}(d,2^b))=b+h^1(\p^{r-1},\LL_{r-1,d}(2^b))$,
and this concludes the proof.
\end{proof}

From now on, we assume that the case of cubics is already solved, i.e. that $\LL_{r,3}(2^n)$ is non-special exept if $r=4$ and $n=7$. The proof of this is completely untied from what follows and it will be discussed in Section \ref{cubics}.  \\We are going to prove that  there exists an upper bound on the number $k$ of nodes such that the linear system $\LL_{r,d}(d-1,2^k)$ is non-special. The proof will be by induction on both $d$ and $r$. Lemma \ref{LF non-special quartics} and Lemma \ref{L_F non speciality trick} provide the starting points of the induction: $d=4$ and  $r=3$ respectively.
Define the integer
$$
k(r):=\left\lceil \frac{1}{r+1}{{r+4}\choose4}\right\rceil-r-1 
$$
\begin{lemma}\label{LF non-special quartics}
Let $r\geq2$. The linear system $\LL_{r,4}(3,2^{k})$, with 
$
k\leq k(r),
$
is non special.
\end{lemma}
\begin{proof}
The proof is by induction on $r$. It suffices to prove the statement for $k(r)$ nodes. For $k<k(r)$, non-speciality  is a  consequence. The base step is the case $r=2$: the system $\LL_{2,4}(3,2^{2})$ is non-special (see Lemma \ref{d-1 point in p2}).
 Consider now the scheme $Z$ given by the union of  the triple point and $k(r-1)<k(r)$ double points. If $\pi \subset \p^r$ is a fixed hyperplane containing the support of $Z$, then the trace of $Z$ with respect to $\pi$ is the scheme $Z\cap \pi$, while the residual scheme is given by a point of multiplicity $2$ and $k(r-1)$ simple points. Thus we get the restriction exact sequence,
$$
0\rightarrow \hat{\LL}:= 	\LL_{r,3}(2^{1+k(r)-k(r-1)},1^{k(r-1)})\rightarrow \LL_{r,4}(3,2^{k(r)})\rightarrow\LL_{\pi}:=\LL_{r-1,4}(3,2^{k(r-1)}).
$$
This gives us the induction on $r$.
The system on the right is non special with virtual dimension at least $-1$, by the inductive hypothesis; moreover the system on the left is non-special and it has virtual dimension $v(\hat{\LL})  \geq  -1$.
 Moreover
$
\dim( \LL_{r,4}(3,2^{k(r)}))=\dim( \LL_{\pi})+\dim (\hat{\LL})+1=v(\LL_{r,4}(3,2^{k(r)})),
$
and this concludes the proof.
\end{proof}

Define now the integers 
$$
k_0(d):=\left\lfloor \frac{d^2+2d-3}{4}\right\rfloor \textrm{ and } h(d):=\left\lfloor \frac{2d+1}{3}\right\rfloor
$$

The reader can easily check that  $k_0(d)-h(d)\leq k_0(d-1)$.

\begin{lemma}\label{d-1 point in p2}
The linear system $\LL_\pi=\LL_{2,d}(d-1,2^{h(d)})$ is non-special, for $d\geq 3$.
\end{lemma}
\begin{proof}
The statement follows by induction on $d$ and by the following restriction sequence:$$0\rightarrow \LL_{2,d-1}(d-2,2^{h(d)-1},1)\rightarrow \LL_\pi\rightarrow \LL_{1,d}(d-1,2).$$
\end{proof}

\begin{lemma}\label{L_F non speciality trick}
Let $d\geq 4$ and    $k\leq k_0(d)$. The linear system $\LL=\LL_{3,d}(d-1,2^k)$ is non-special.
\end{lemma}
\begin{proof}
We prove the statement by induction on $d$ for a collection of $k_0(d)$ points. The base step is the case $\LL_{3,4}(3,2^5)$ that is non-special by Lemma \ref{LF non-special quartics}. 
The induction is given by specializing $p$ and 
$h(d)$ nodes on a general plane $\pi\subseteq\p^3$: 
$$
\hat{\LL}:=\LL_{3,d-1}(d-2,2^{k_0(d)-h(d)},1^{h(d)})\rightarrow \LL\rightarrow\LL_{\pi}:=\LL_{2,d}(d-1,2^{h(d)}).
$$
the kernel system  $\hat{\LL}$ is non-special by induction. 
Moreover it has positive virtual dimension.
Finally,
$\dim(\LL)  = \dim(\hat{\LL})+\dim(\LL_{\pi})+1=e(\LL)$
and this concludes the proof.
\end{proof}

Now, we prove a non-speciality result for  linear systems of hypersurfaces of degree $d$ of $\p^r$  with a point of multiplicity $d-1$ and $k$ general nodes in full generality. 
To this,   define the number
$$
k(r,d):=\left\lfloor \frac{1}{r+1}{{r+d}\choose r}-\frac{1}{r+1}{{r+d-2}\choose r}\right\rfloor -(r-2),
$$
for every $r\geq3$ and $d\geq 5$. We want to prove that the linear system $\LL_{r,d}(d-1,2^{k})$ is non special, if $k\leq k(r,d)$.
\begin{rmk}
Notice that $k(3,d)$ is equal to the number $k_0(d)$ defined in Lemma \ref{L_F non speciality trick}, so that result can be employed as the base step of the induction on $r$. Moreover, being $k(r,4)\leq k(r)$, the linear system $\LL_{r,4}(3,2^k)$  is non-special by Lemma \ref{LF non-special quartics}, so this can be used as starting step of the induction on $d$.
\end{rmk}

As in the case  $r=3$, the trick will be to specialize $k(r-1,d)$ nodes on an hyperplane $\pi\cong \p^{r-1}$ containing the support of $p$  as follows:
\begin{eqnarray}\label{restr seq LF trick}
0\rightarrow \hat{\LL}\rightarrow\LL=\LL_{r,d}(d-1,2^{k(r,d)})\rightarrow\LL_{\pi}=\LL_{r-1,d}(d-1,2^{k(r-1,d)}),
\end{eqnarray}
where the kernel system is $\hat{\LL}=\LL_{r,d-1}(d-2,2^{k(r,d)-k(r-1,d)},1^{k(r-1,d)}).$  

\begin{prop}\label{per LF non special k(r,d)}
The linear system $\LL_{r,d}(d-1,2^{k})$, with $k \leq k(r,d)$ and $d\geq 4$, is non-special and it has virtual dimension at least $-1$.
\end{prop}
\begin{proof}
Consider the restriction exact sequence in (\ref{restr seq LF trick}): $\LL_{\pi}$ is non-special by induction on $r$, and $v(\LL_{\pi})\geq-1$; moreover  $\hat{\LL}$ is non-special by induction on $d$ and by the fact that  $ k(r,d)-k(r-1,d)\leq k(r,d-1)$ as one can easily check; moreover  $v(\hat{\LL})\geq-1$.
Finally, $\dim(\LL)=\dim(\hat{\LL})+\dim(\LL_{\pi})+1=e(\LL)$ and this completes the proof. 
\end{proof}

\begin{rmk}
Lemma \ref{LF non-special quartics} provides an upper bound for the number of double points which is bigger than the one we need for the base step of the induction on the degree used in the proof of Proposition \ref{per LF non special k(r,d)}. Nevertheless  $k(r)$ is exactly the number of nodes that we will specialize on the component $\F$ in the proof of Alexander-Hirschowitz Theorem in degree four (Section \ref{quartics}).
\end{rmk}


\subsection{First transversality lemma}\label{transIdeg}
In this section we will show that it is possible to choose a specialization of the nodes such that $\LL_\F$ cuts the complete series on $R=\F\cap\p$, for $d\geq5$; in this way we also get transversality of the restricted systems.
\begin{lemma}[Transversality Lemma I]\label{transversality}
Performing a $(1,b)$-degeneration with $b\in\Z$ such that the system $\hat{\LL}_\F\cong \LL_{r-1,d}(2^b)$ has dimension $\hat{l}_\F=\dim(\LL_{r-1,d}(2^b))=v(\LL_{r-1,d}(2^b))\geq -1$, then the restricted systems $\RR_{\p}$ and $\RR_{\F}$ are transversal in $|\OO_{\p^{r-1}}(d-1)|$.
\end{lemma}
\begin{proof}
 Notice that $b\leq \left\lfloor \frac{1}{r} {{r+d-1}\choose{r-1}}\right\rfloor\leq k(r,d)$, therefore the system $\LL_{\F}$ is non-special, by Proposition \ref{per LF non special k(r,d)}. Let $p_i=[p_{i,0},\dots,p_{i,r}]\in \F$, $i=1,\dots,b$, be the general double points. 
The restricted system on $R$ is the complete linear system of  hypersurfaces of $\p^{r-1}$ of degree $d-1$ containing $b$ simple general points $q_i=[p_{i,0},\dots,p_{i,r-1}]$, $i=1,\dots,b$, which are the traces on the exceptional divisor $R$ of the $b$ lines through the $(d-1)$-point $p$,
 that we blew up, and the  points $p_1,\dots, p_b$: $\RR_{\F}=\LL_{r-1,d-1}(1^b)$: indeed
a local computation that we omit proves that $\RR_\F\subseteq \LL_{r-1,d-1}(1^b)$, moreover $r_{\F}  =  l_{\F}-\hat{l}_{\F}-1 ={{r+d-2}\choose{r-1}}-1-b.$\\
Now, if $\LL_{\p}$ is empty, transversality is trivial, being $\RR=\emptyset$.  So, assume  that $\LL_{\p}\neq \emptyset$. The general section of $\RR$ is  a section of $\RR_\p$ that vanishes at $q_1,\dots,q_b$ which are in general position in $R$. Therefore $\RR_\p$ and $\RR_\F$ intersect transversally in $R$ and $\dim(\RR)=\max\{-1,\dim(\RR_\F)-b\}$.
\end{proof}


\subsection{Quartics}\label{quartics}
This section is devoted to the analisys of the case $d=4$.
\subsubsection{Quartics in $\p^3$.}
If $n > 9$ the  system is  empty, being $\dim(\LL_{3,4}(2^9))=0$ (see Section \ref{specials}). For $n=8$, we prove non-speciality of the corresponding linear system performing a $(1,4)$-degeneration: $\LL_{\F}={\LL}_{3,4}(3,2^4) , \hat{\LL}_{\F}={\LL}_{3,4}(4,2^4)\cong {\LL}_{2,4}(2^4), 
\LL_{\p}={\LL}_{3,3}(2^4) , \hat{\LL}_{\p}={\LL}_{3,2}(2^4)=\emptyset$.
The system $\LL_{\p}$ is non-special (as we will see in Section \ref{cubics}), while the system $\LL_{\F}$ is non-special by Proposition \ref{per LF non special k(r,d)}; $\RR_{\F}$ is the complete series $\LL_{2,3}(1^4)$ and the restricted systems intersect transversally (see the proof of  Lemma \ref{transversality}). Hence
$l_0  =  \dim(\RR)+\hat{l}_{\p}+\hat{l}_{\F}+2 =  v(\LL_{3,4}(2^8))$.
It follows  that also the system of quartic surfaces of $\p^3$ having $n$ nodes, with $n<8$, is non-special.

\subsubsection{Quartics in $\p^4$.}
The systems of quartics with $n$ nodes  in $\p^4$, for  $n>14$, is  empty, being $\dim(\LL_{4,4}(2^{14}))=0$ (see Section \ref{specials}).
Performing a $(1,8)$-degeneration of $\p^4$, we prove that the system $\LL_{4,4}(2^{13})$ is non-special, as in the case of $\p^3$.
As consequence, for  a smaller number of general nodes, the system of quartics of $\p^4$ is non-special.

\subsubsection{Quartics in $\p^r$, $r\geq5$.}
Let $n^-(r,4)\leq n \leq n^+(r,4)$. We will prove non-speciality of the linear  system $\LL_{r,4}(2^n)$, 
 performing a $(1,n-r-1)$-degeneration.
The system $\hat{\LL}_{\F}=\LL_{r,4}(4,2^{n-r-1})$ is empty, in fact it has dimension 
$\dim(\LL_{r-1,4}(2^{n-r-1}))=-1$ (use induction on $r$ and check that the virtual dimension is negative);
furthermore the system $\LL_{\F}=\LL_{r,4}(3,2^{n-r-1})$ is non-special by Lemma \ref{LF non-special quartics}, being $n-r-1\leq k(r)$.
The system $\LL_{\p}=\LL_{r,3}(2^{r+1})$ is non-special, under the assumption that Alexander-Hirschowitz Theorem holds for cubics (see Section \ref{cubics}).
The kernel system 
$
\hat{\LL}_{\p}=\LL_{r,2}(2^{r+1})
$
of quadric hypersurfaces with $r+1$ double points is empty, therefore the restriction map $\LL_{\p} \hookrightarrow \RR_{\p}\subseteq |\OO_{\p^{r-1}}(3)|$ is injective. 
 If a cubic has $k$ nodes, then it must contains all the ${k\choose2}$ lines joining the points. Consequentely, when we restrict to the hyperplane $R$, the image of the cubics in $\LL_{\p}$ must contain the traces of these lines as base points; so we get
$
\RR_{\p}\subseteq \LL_{r-1,3}(1^{{r+1}\choose2}).
$
Actually, these ${{r+1}\choose 2}$ points give  independent conditions, and therefore $\RR_{\p}$ is the complete series.
\begin{prop}[Transversality for quartics]
In the setting of above, the system $\RR_{\p}$ is the complete linear system of cubics of $R$ with ${{r+1}\choose 2}$ base points and $\dim(\RR_\p)={{r+2}\choose{r-1}}-1-{{r+1}\choose2}$.
\end{prop}
\begin{proof}
We have to prove that the ${{r+1}\choose 2}$ points on $R$, traces of the lines joining the $r+1$ nodes $p_1,\dots,p_{r+1}$ specialized on the component $\p$, impose independent conditions. If we prove that this is true for quadrics, it will be true  for cubics.
The case $r=3$ is easy: let $p_1,\dots,p_4$ points in $\p\cong \p^3$: three of them, say $p_1,p_2,p_3$, span a plane $\pi$, which cuts a line $\pi'$ on  $R\cong \p^2$; on this line we will have the three distinct points given as traces of the three lines $<p_i,p_j>, \ i\neq j, \ i,j=1,2,3$. The line $\pi'$ splits off the system of conics through these three points, thus
$$
\LL_{2,2}(1^6)=\pi'+\LL_{2,1}(1^3),
$$ 
where the three base points of the system on the right are the projection of $p_1,p_2,p_3$ from $p_4$ on $R$ and they will not lie on a line, by generality. So our system is empty.
For $r>3$, apply induction and use the same argument. 
\end{proof}
From this follows that $\RR_\p$ and $\RR_\F$ intersect properly in $R$, being $\RR_\p$ the complete series.
In particular, the system $\RR$ consists of the sections of $\RR_\F$ vanishing at the ${{r+1}\choose2}$  base points of $\RR_\p$, thus 
$
\dim( \RR)=\max\left\{l_\F-{{r+1}\choose2},-1\right\}=e(\LL)
$
and  we conclude applying Formula (\ref{formula l_0}).

\begin{rmk}
This discussion does not apply if $r=3,4$. Indeed the kernel on the component $\F$ would be isomorphic to $\LL_{2,4}(2^5)$ and $\LL_{3,4}(2^9)$ respectively, that are special and in particular nonempty.
\end{rmk}


\subsection{Proof of Theorem \ref{A-H}, part I}

The goal of this section is to apply the above degeneration technique  to linear systems of hypersurfaces of $\p^r$, with $r\geq3$, of degree $d\geq5$ with a collection of $n$ nodes in general position, with $n^- \leq n \leq n^+$.
 Let us set 
\begin{eqnarray}\label{beta}
b_0(r,d):=\frac{1}{r}{{r+d-1}\choose{r-1}}=\left\lfloor \frac{1}{r}{{r+d-1}\choose{r-1}}\right\rfloor +\frac{\beta}{r}, \ \ \ \ \beta \in \{0,\dots,r-1\},
\end{eqnarray}
and assume that $b_0\in \Z$, i.e. that $\beta=0$. Perform a $(1,b)$-degeneration of $\p^r$ and $\LL$, with $b=b_0$.

\begin{prop}\label{B in Z}
Same  notation of above.  If the linear systems $\LL_{r-1,d}(2^b)(\cong\hat{\LL}_\F)$, $\LL_\p=\LL_{r,d-1}(2^{n-b})$ and $\hat{\LL}_\p=\LL_{r,d-2}(2^{n-b})$ are non-special, then the linear system $\LL_{r,d}(2^n)$ is non-special.
\end{prop}
\begin{proof}
The system $\hat{\LL}_{\F}\cong \LL_{r-1,d}(2^{b})$ has dimension $-1$. The system $\LL_{\F}=\LL_{r,d}(d-1,2^b)$ is non-special, by Proposition
\ref{per LF non special k(r,d)}            and  cuts the complete series $\RR_{\F}=\LL_{r-1,d-1}(1^b)$ on $R$, by Lemma \ref{transversality}.  Furthermore the restricted systems intersect transversally.
The system on the component $\p$ is non special, by assumption and it is  nonempty, because one computes $v_{\p}>-1$.
The kernel system $\hat{\LL}_{\p}$ is non-special and empty, because $\hat{v}_\p\leq-1$.
Moreover the dimension of the intersection $\RR$ of the restricted systems on $R$ is
$
\dim (\RR)=\max\{l_{\p}-b,-1\}= e(\LL).
$
 Now, using formula (\ref{formula l'_0}), we get 
$
l_0  =  \dim(\RR)+\hat{l}_{\p}+\hat{l}_{\F}+2=\dim(\RR)=e(\LL).
$
 Therefore, by  upper-semicontinuity, the system $\LL$ is non-special.
\end{proof}


\section{The second degeneration}\label{IIdeg}

Specializing $b$ double points  on  $\F$ and the others on $\p$ does not cover all the cases. Trying to prove  non-speciality of  a given linear system $\LL$, in some cases
 we are not  able to find an integer $b$ such that the limiting system $\LL_0$ has dimension equal to  $e(\LL)$. 
 In those cases a arithmetic obstruction prevents us from finding such a $b$.
 For example when $\LL$ is expected to be empty, namely if $n=n^+=\left\lceil \frac{1}{r+1}{{r+d}\choose{r}}\right\rceil=\frac{1}{r+1}{{r+d}\choose r}+\frac{l}{r+1}$, $l\in\{0,\dots,r\}$ and $v(\LL)=-1-l$, we would like to find a specialization  such that both kernel systems and also the intersection of the restricted systems are empty. 
The minimal integer $b$ such that $\hat{\LL}_{\F}$ is empty is 
$
b=\left\lceil \frac{1}{r}{{r+d-1}\choose{r-1}}\right\rceil=\frac{1}{r}{{r+d-1}\choose{r-1}}+\frac{l'}{r}, l'\in\{0,\dots,r-1\}.
$
But we would have a problem with the dimension of the intersection of the restricted systems $\RR$ on $R$ (which we wish to be empty); indeed
$
\dim(\RR)=\max\{-1,l_{\p}-b\}  =\max\{-1,   -1-l + l'\}
$
and we are not able to check if $l'\leq l$.
On the other hand, if we choose 
$
b=\left\lfloor \frac{1}{r}{{r+d-1}\choose{r-1}}\right\rfloor,
$
then $\hat{\LL}_{\F}$, which has dimension $\dim(\LL_{r-1,d}(2^b))$, is nonempty.

Hence we use another approach in order to overcome the problem. It consistes in degenerating the system $\LL_0$ on the central fiber $X_0$ to a system $\LL'_0$ such that some of the points  of $\F$ \emph{approach} $R$.\\
 Let $\Delta'$
be a complex disc around the origin. Consider the trivial family $\mathcal{Z}=Z\times\Delta'\rightarrow \Delta'$
 with reducible fibers $Z_s=\F_s\cup\p_s$, where $\F_s=\F$ is isomorphic to $\p^r$ blown up at a point, $\p_s=\p$ is isomorphic to $\p^r$ and $\F_s\cap\p_s=R_s\cong\p^{r-1}$, 
for every $s\in\Delta'$. 
Consider on $Z_s$, $s\neq0$, the linear system $\LL'_s:=\LL_0$, where $\LL_0$ is the flat limit of $\LL_t=\LL_{r,d}(2^n)$, with respect to the  first degeneration. Such a system is given by two linear systems $\LL'_{\p_s}$ and 
 $\LL'_{\F_s}$ on the two components 
that agree on the intersection $R_s$. The system on $\p_s$ (on $\F_s$) restricts to a system $\RR'_{\p_s}$ 
($\RR'_{\F_s}$ respectively) and the  kernel is $\hat{\LL}'_{\p_s}$ 
($\hat{\LL}'_{\F_s}$ respectively).
We have the following identities:
$$
\hat{\LL}'_{\F_s}=\hat{\LL}_{\F}, \ \hat{\LL}'_{\p_s}=\hat{\LL}_{\p}, 
\  \RR'_{\F_s}=\RR'_{\F}, \ \RR'_{\p_s}=\RR'_{\p}, \textrm{ for } s\neq0
$$

Now, let 
$\beta \in \N$ such that 
 $\beta\leq b$.
 Consider the scheme on the central fiber given by
$n-b$ double points in $\p_0\setminus R_0$,
 $b-\beta$ in $\F_0\setminus R_0$ 
  and $\beta$ in $R_0$:
 we can consider these
 nodes as limit of the $n$ general nodes in $Z_s$ ($n-b$ in $\F_s$ and $b$ in $\p_s$, $s\neq0$). So, on the central fiber $Z_0$ the systems $\LL'_{\p_0}$, $\hat{\LL}'_{\p_0}$ and $\LL'_{\F_0}$ are still the same, while 
$$
\hat{\LL}'_{\F_0}=\LL_{r,d}(d,2^{b-\beta},1^{\beta})\cong\LL_{r-1,d}(2^{b-\beta},1^{\beta})
$$
with the following restriction sequences:
$$
\begin{array}{lllllll}
0& \rightarrow & \hat{\LL}'_{\p_0} & \rightarrow & \LL'_{\p_0} & \rightarrow & \RR'_{\p_0}\subseteq|\OO_{\p^{r-1 }}(d-1)| \\
0& \rightarrow & \hat{\LL}'_{\F_0} & \rightarrow & \LL'_{\F_0} & \rightarrow & \RR'_{\F_0}\subseteq\LL_{r-1,d-1}(2^{\beta}) 
\end{array}
$$
We denote by $\hat{v}'_{\p_0}$, $v'_{\p_0}$, $\hat{v}'_{\F_0}$, $v'_{\F_0}$ and $\hat{l}'_{\p_0}$, $l'_{\p_0}$, $\hat{l}'_{\F_0}$, $l'_{\F_0}$ the virtual and the actual dimensions. Let $\RR'_0:=\RR'_{\p_0}\cap\RR'_{\F_0}$.
As in Section \ref{Ideg}, we obtain a recursive formula for the dimension of $\LL'_0$:
\begin{eqnarray}\label{formula l'_0}
l'_0=\hat{l}'_{\p_0}+\hat{l}'_{\F_0}+\dim(\RR'_0)+2.
\end{eqnarray}

\begin{lemma}
Keeping the same notations as above, if there are integers $b,\beta$ such that $l'_0=e(\LL)$, then $\LL$ is non-special
\end{lemma}


\subsection{Second transversality Lemma}\label{transIIdeg}
The intersection $\RR'_0$ is contained in the linear system of those divisors of $\RR'_{\p_0}$ that are singular at  $\beta$ further general points of $R_0$, the ones imposed by $\RR'_{\F_0}$, and satisfying the remaining matching conditions.  Let us denote by
$$
{\hat{\LL'}}^{m}_{\p_0} \subseteq \hat{\LL}'_{\p_0}, \ {\LL'}_{\p_0}^{m}\subseteq \LL'_{\p_0} \textrm{ and } {\RR'}_{\p_0}^m \subseteq \RR'_{\p_0},
$$
the systems defined by the matching conditions imposed by $\RR_\F$ to $\RR_\p$.
The first step is to prove that if we impose our $\beta$ nodes to $\RR'_{\p_0}$, the resulting system $\bar{\RR}'_{\p_0}:=\RR'_{\p_0}(2^{\beta})$ is non-special, i.e. $
\dim(\bar{\RR}'_{\p_0})=\dim(\RR'_{\p_0})-r\beta.$ 
Notice that ${\RR'}_{\p_0}^m\subseteq\bar{\RR}'_{\p_0}\subseteq \RR'_{\p_0}$.
Define  $\bar{\LL}'_{\p_0}:=\LL_{r,d-1}(2^{n-b+\beta})\subset\LL'_{\p_0}$ to be the linear system of hypersurfaces of degree $d-1$  of $\p_0$ with $n-b$ general nodes on $\p_0$ and $\beta$ general nodes on $R_0\subseteq\p_0$. Recall that $R_0$ is a general hyperplane for $\p_0$ and notice that if  $\beta<r$ then the $n-b+\beta$ nodes are in general position in $\p_0$.

\begin{lemma}[Transversality Lemma II]\label{transversality 2}
In the above construction, assume that $h^0(\p_0, \hat{\LL}'_{\p_0})=0$, $h^1(\p_0, \hat{\LL}'_{\p_0})>0$, that $h^1(\p_0,\bar{\LL}'_{\p_0})=0$ and that $\beta <r$.
 Then the linear system $\bar{\RR}'_{\p_0}$ is non-special.
\end{lemma}
\begin{proof}
Restricting $\bar{\LL}'_{\p_0}$ to $R_0\cong\p^{r-1}$, we get the following  exact sequence:
$$
0\rightarrow\hat{\bar{\LL}}'_{\p_0}:=\LL_{r,d-2}(2^{n-b},1^{\beta})\rightarrow\bar{\LL}'_{\p_0}\rightarrow\bar{\RR}'_{\p_0}
$$
Notice that
 $h^0(\p_0,\hat{\bar{\LL}}'_{\p_0})=h^0(\p_0, \hat{\LL}'_{\p_0})=0$,  $h^0(\p_0,\bar{\LL}'_{\p_0})=h^1(\p_0, \LL'_{\p_0})-(r+1)\beta$ and that
 $h^1(\p_0,\hat{\bar{\LL}}'_{\p_0})=h^1(\p_0, \hat{\LL}'_{\p_0})+\beta$.
We have the following commutative diagram.
\begin{displaymath}
\xymatrix{ 
 &0 &0 & & & \\
 & K^{(r+1)\beta} \ar[u]& V \ar[u] &0 & & \\
  & & & & & \\
0 \ar[r]   & H^0(\p_0,\LL'_{\p_0}) \ar[r] \ar[uu] & H^0(R_0,\RR'_{\p_0}) \ar[r] \ar[uu]  & H^1(\p_0,\hat{\LL}'_{\p_0})  \ar[r] \ar[uu] & 0 & \\
0 \ar[r]   & H^0(\p_0,\bar{\LL}'_{\p_0}) \ar[r] \ar@{^{(}->}[u] & H^0(R_0,\bar{\RR}'_{\p_0}) \ar[r] \ar@{^{(}->}[u] & H^1(\p_0,\hat{\bar{\LL}}'_{\p_0})  \ar[r] \ar[u] & 0  & \\
  &0 \ar[u]  & 0 \ar[u]  & K^{\beta}\ar@{^{(}->}[u] 
  \ar `^u/10pt[rr] `^l/10pt[uuu] `_u/10pt[lll] `/10pt[uuuull] [lluuuu]
  &  & \\
 & & &0\ar[u] & &
 }
\end{displaymath}
It follows that $\dim (V)=r\beta$.
Hence $H^0(R_0,\bar{\RR}'_{\p_0})$ has codimension equal to $\dim( V)=r\beta$ in  $H^0(R_0,\RR'_{\p_0})$ and, at the level of linear systems, $\dim(\bar{\RR}'_{\p_0})=\dim({\RR'}_{\p_0})-r\beta$. Therefore the $\beta$ nodes impose $r\beta$ linearly independent conditions to $\bar{\RR}'_{\p_0}$ and
this concludes the proof.
\end{proof}

\begin{cor}
In the setting of above, the $\beta$ nodes specialized on $R_0$ give linearly independent conditions to the matching system ${\RR'}_{\p_0}^m$.
\end{cor}


\subsection{Proof of Theorem \ref{A-H}, part II}
In this section we analyse the cases such that $b_0(r,d)\notin\Z$, i.e. the cases for which the first degeneration argument does not suffice.
Let $\LL'_{s}=\LL'_{\F_s}\times_{\RR'_{\F_s}\cap\RR'_{\p_s}}\LL'_{\p_s}$ be the system on the general fiber, which corresponds to the limit system of $\LL_{r,d}(2^n)$, with respect to the first degeneration. 
Let $b_0(r,d)$ and $\beta$ be as defined in (\ref{beta}); choose
$$
b=\frac{1}{r}{{r+d-1}\choose{r-1}}-\frac{\beta}{r}+\beta \in \Z.
$$

\begin{rmk}
Notice that if $r$ and $d$ are such that $\beta=0$, then we are in the case $b_0\in\Z$ and Proposition \ref{B in Z} applies, with no need of making a second degeneration.
\end{rmk}

From now on, we assume that the systems  $\LL_{r-1,d}(2^{b-\beta})$, $\LL_{r,d-1}(2^{n-b+\beta})$ and $\LL_{r,d-2}(2^{n-b})$ are non-special, for points in general position.
Consider the following exact sequence on  $\F_0$:
\begin{eqnarray}\label{F_0 complete series}
0 \rightarrow\hat{\LL}'_{\F_0}=\LL_{r,d}(d,2^{b-\beta},1^{\beta})\rightarrow \LL'_{\F_0}\rightarrow \RR'_{\F_0}\subseteq\LL_{r-1,d-1}(1^{b-\beta},2^{\beta})
\end{eqnarray}
The kernel system is empty and $h^1(\F_0, \hat{\LL}'_{\F_0})=0$. The system $\LL'_{\F_0}=\LL_{r,d}(d-1,2^b)$ is non-special by Proposition \ref{per LF non special k(r,d)}, in fact one checks that $b\leq k(r,d)$, for $r\geq 3$, $d\geq5$. Therefore $\LL'_{\F_0}$ is non-special in all  cases we are interested in.
Furthermore it cuts the complete series on $R_0$, namely
$
\RR'_{\F_0}=\LL_{r-1,d-1}(1^{b-\beta},2^{\beta}),
$
in fact the $b-\beta$ simple points (trace on $R_0$ of the lines through the $b-\beta$ double points and the $(d-1)$-point) are base points (see Lemma \ref{transversality}).
Moreover the system $\hat{\LL}'_{\p_0}=\LL_{r,d-2}(2^{n-b})$ is empty; indeed it is non-special by assumption and $\hat{v}'_{\p_0}\leq-1$.
 
The linear system  $\bar{\LL}'_{\p_0}=\LL_{r,d-1}(2^{n-b+\beta})\subset\LL'_{\p_0}$ on $\p_0$ is non-special by assumption, moreover $\dim(\bar{\LL}'_{\p_0})\geq  -1$, for $r\geq3$ and $d\geq5$. The hypotheses of Proposition \ref{transversality 2} are satisfied, hence $\RR'_{\p_0}$ and $\RR'_{\F_0}$ intersect transversally on $R_0$ and we get
$$
\dim(\RR')=\dim({\RR'}_{\p_0}^m)=\max\{-1,l'_{\p_0}-r\beta-(b-\beta)\}=e(\LL)
$$

\begin{prop}\label{B notin Z}
In the above notation, assume that the sistems $\LL_{r-1,d}(2^{b-\beta})(\cong\hat{\LL'}_{\F_0})$, $\bar{\LL}'_{\p_0}$ and $\hat{\LL}'_{\p_0}$ are non-special. Then the linear system $\LL_{r,d}(2^n)$ is non-special.
\end{prop}
\begin{proof}
Following the argument of this section we get $l'_0=\dim(\RR')+\hat{l}'_{\p_0}+\hat{l}'_{\F_0}+2=\dim(\RR')=e(\LL)$. 
\end{proof}

Putting together Propostion \ref{B in Z} and Proposition \ref{B notin Z}, the proof  of Theorem \ref{A-H} for $d\geq4$ is now complete.


\section{Cubics}
\label{cubics}

The techniques introduced in the previous sections do not work in the case of cubics, because the limiting system on the exceptional component $\p$ of the central fiber of a $(1,b)$-degeneration is a linear system of quadrics with nodes which is special. We will prove non-speciality of $\LL_{r,3}(2^n)$, for $r\geq3$, $r\neq4$ by induction on $r$, with a different  degeneration argument.\\
The starting point is the linear system $\LL_{3,3}(2^5)$ of cubic surfaces of $\p^3$, which is empty as expected. Indeed, if we restrict it to a plane $\pi$ and if we specialize three nodes on it, 
we get the following sequence:
$$
0\rightarrow\LL_{3,2}(2^2,1^3)\rightarrow\LL_{3,3}(2^5)\rightarrow\LL_{3,3}(2^5)_{|\pi}\subset \LL_{2,3}(2^3).
$$
If a cubic has two double points,  it must vanish identically on the line joining them. This line meets $\pi$ at a point, so
$
\LL_{3,3}(2^5)_{|\pi}\subseteq \LL_{2,3}(2^3,1)=\emptyset.
$
Moreover the kernel $\LL_{3,2}(2^2,1^3)$ is empty, and this concludes the proof in the case of $\p^3$.

We will study the linear system $\LL_{r,3}(2^n)$, for $r\geq5$. Let $l \in\Z$ such that $r=3k+l$, with $k\in\{0,1,2\}$. Notice that $n^-(r,3)$ is an integer if and only if $k=0,1$ and define
$$
\gamma(r):=\begin{cases}
0& \textrm{if } r\equiv 0,1\pmod 3\\
l+1& \textrm{if } r\equiv 2\pmod3.
\end{cases}
$$
We will prove the following.
\begin{thm}
\label{AH 3}
The linear system $\LL:=\LL_{r,3}(2^{n^-(r,3)},1^{\gamma(r)})$ is empty, for $r\geq 5$. 
\end{thm}

\subsubsection{The degeneration construction}
We will perform a degeneration of $\p^r$ and simultaneously of $\LL$, by blowing up a $3$-codimensional subspace of the central fiber.
Let us consider the trivial family $\mathcal{Y}=\p^r\times\Delta\rightarrow\Delta$, where $\Delta$ is a complex disc with center at the origin. Let $Y_0$ be the central fiber.  Choose a general linear subspace $L\subset Y_0=\p^r$  of codimension $h$: 
$
\mathcal{N}_{L|\mathcal{Y}}=\OO_{L}(1)^{\oplus h}\oplus\OO_{L}
$ 
is the normal sheaf of $L$ in $\mathcal{Y}$. Blowing up $L$ in the  family, we obtain a new family $\mathcal{X}$, with maps $\pi_1:\mathcal{X}\rightarrow \p^r$ and $\pi_2:\mathcal{X}\rightarrow \Delta$ and a reducible central fiber $X_0$  which is the union of 
 the strict transform $V$ of $\p^r$, i.e. $\p^r$ blown up along $L$, and the exceptional divisor $T$, which is isomorphic to $\p(\mathcal{N}^{\ast}_{L|\mathcal{Y}})\cong\p(\OO_{\p^{r-h}}(1)^{\oplus h}\oplus\OO_{\p^{r-h}}(2))$.
This variety of dimension $r$ is a $\p^h$-bundle over $L\cong\p^{r-h}$ with the natural map $p:T\rightarrow L$.
The intersection of the two components of $X_0$  is a $(r-1)$-dimensional subvariety $Q$ of degree $h$:
$Q=\p(\OO_{\p^{r-h}}(1)^{\oplus h})  \cong \p^{r-h}\times \p^{h-1}$;
it is the exceptional divisor of the blow up of $L$ in the central fiber.
The Picard group of $V$ is generated by the hyperplane class $H_V$, which corresponds to the line bundle $\OO_{V}(1)$ pull back of $\OO_{\p^r}(1)$, and by the divisor $Q$.
The Picard group of $T$ is generated  by $\pi:=p^{\ast}(\OO_{L}(1))$ and by $Q$; so the $\OO(1)$-bundle of  $T\cong\p(\OO_{\p^{r-h}}(1)^{\oplus h}\oplus\OO_{\p^{r-h}}(2))$ is of the form $H_T=Q+2\pi$.

Now, consider the line bundle $\OO_{\mathcal{X}}(3)=\pi^{\ast}_1\OO_{\p^r}(3)$. It restricts to $\OO_{\p^r}(3)$ on the general fiber; while on the central one we have:
$$
|\OO_{\mathcal{X}}(3)_{|T}|=|3\pi| \textrm{ and } |\OO_{\mathcal{X}}(3)_{|V}|=|3H_V|.
$$
If $r\neq 7$, choose $h=3$ and twist by $-T$: on the general fiber we do not make any change, while on the special one we get:
$$
 |3\pi+Q| \textrm{ on }T \textrm{ and }|3H_V-Q| \textrm{ on } V.
$$
Consider the linear system of cubics on the general fiber  $\LL_t:=\LL=\LL_{r,3}(2^{n^-(r,3) },1^{\gamma(r)})$, 
which has virtual dimension $-1$, for every $r$. Specialize $n^-(r-3,3)$ nodes and $\gamma(r-3)$ simple points on the component $T$:
$$
\LL_{T}=|3\pi+Q-2^{n^-(r-3,3) }-1^{\gamma(r-3)}| \textrm{ and }
\LL_{V} =|3H-Q-2^{r+1}-1^{\gamma(r)-\gamma(r-3)},
$$ 
where $\gamma(r)-\gamma(r-3)\in\{0,1\}$.
The system $\LL_V$ is isomorphic to the linear system of cubic hypersurfaces of $\p^r$ containing a $3$-codimensional subspace $L$, being singular at $r+1$ general points, and passing through $\gamma(r)-\gamma(r-3)$ general  points; we will use the following notation: $\LL_V\cong\LL_{r,3}(L,2^{r+1},1^{\gamma(r-3)}).$\\
If we restrict the two linear systems to the intersection $Q$, we obtain as kernels the following systems:
\begin{align*}
\hat{\LL}_{T}&= |3\pi-2^{n^-(r-3,3)}-1^{\gamma(r-3)}|\cong \LL_{r-3,3}(2^{n^-(r-3,3)},1^{\gamma(r-3)}) \\
\hat{\LL}_V&=|3H-2Q-2^{r+1}-1^{\gamma(r)-\gamma(r-3)}|
\end{align*}
The motivation of this choice is that in this way the kernel $\hat{\LL}_T$ is known to be empty, applying induction from $r-3$ to $r$, if $r\neq7$.

\subsubsection{Emptyness of $\hat{\LL}_V$}
The kernel system of the component $V$ is isomorphic to the linear system of cubic hypersurfaces of $\p^r$ that are singular along a $3$-codimensional subspace $L$ and at $r+1$ general points, and with $\gamma(r)-\gamma(r-3)$ additional base points: $\hat{\LL}_V\cong\LL_{r,3}(2L,2^{r+1},1^{\gamma(r)-\gamma(r-3)}).$
Notice that  $\LL_{r,2}(2L, 2)$  has dimension $2$; this is easy and can be left to the reader.

\begin{prop}\label{empty con 2L}
The system $\LL_{r,3}(2L,2^{r+1})$ is empty, for $r\geq3$.
\end{prop}
\begin{proof}
In the first case ($r=3$) the system is $\LL_{3,3}(2^5)$ that is empty. For $r\geq4$,  the statement follows by induction on $r$ and by the sequence
$$
0\rightarrow\LL_{r,2}(2L,2,1^r)\rightarrow\LL_{r,3}(2L,2^{r+1})\rightarrow\LL_{r-1,3}(2L',2^r)
$$
where $L'\cong\p^{r-4}$ is the intersection of $L$ with  the restricting hyperplane. 
\end{proof}

From this proposition in particular follows  the emptyness of $\hat{\LL}_V$, for $r\geq 5$, $r\neq7$.

\subsubsection{Matching systems}

Let $p_1,\dots,p_t$, with $t=n^-(r-3,3)$, be the nodes specialized on $T$. Each of them lies on a distinct fiber of the ruling of $T$: say $p_i\in f_i\cong \p^3$. This implies that each of the sections of $\LL_T$ must contain $f_1,\dots,f_t$. Therefore the sections of ${\LL_{T}}_{|Q}$ must contain $t$ distinct planes $\sigma_i={f_i}_{|Q}\cong \p^2$, each of them imposing $3$ linear conditions on it. \\
The sections of ${\LL_V}_{|Q}$ must agree with those of ${\LL_T}_{|Q}$. Define $\LL_{V}^m\subseteq\LL_{V}$ and $\hat{\LL}_{V}^m\subseteq\hat{\LL}_{V}$ to be the linear systems on $V$ defined by the matching conditions. Similarly, let $\LL_{T}^m\subseteq\LL_{T}$ and $\hat{\LL}_{T}^m\subseteq\hat{\LL}_{T}$ be the corresponding systems on the exceptional component. 
The system $\LL_{V}^m$ is  the linear system of cubic hypersurfaces of $\p^r$ which contain a linear subspace $L$ of codimension $3$ and  which are singular at $n^-(r,3)$ nodes, such that $n^-(r-3,3)$ of them are supported on $L$ and $r+1$ are general in $\p^r\setminus L$ and which pass through $\gamma(r)-\gamma(r-3)$ additional general points:
$$
\LL_{V}^m\cong \LL_{r,3}(\{L,2^{n^-(r-3,3)}\},2^{r+1},1^{\gamma(r)-\gamma(r-3)}),\ \ r\neq7.
$$
We will use the notation $\{L,2^t\}$ for the scheme given by a subspace $L$ and $t$ general nodes supported on it.
We will prove that the matching system $\LL_{V}^m$ is empty, for $r\geq5$ by induction  from $r-3$ to $r$, starting from the cases $r=5,6,7$. This proof is very similar to the one of M. C. Brambilla and G. Ottaviani in \cite{BO}, Section 5.
We need two preliminary results.

\begin{prop}\label{prop K_2}
The system   $\mathscr{K}_{2}(r):=\LL_{r,3}(\{L_1,2^3\},\{L_2,2^3\},\{L_3,2^3\})$, with $L_1,L_2,L_3\cong\p^{r-3}$ three general subspaces of $\p^r$, is empty for $r\geq 6$.
\end{prop}
\begin{proof}
For $r=6$ it suffices to make an explicit computation. 
For $r\geq7$, we prove the statement by induction from $r-1$ to $r$. Choose a general hyperplane of $\p^r$: it intersects $L_i$ in a subspace $L'_i$ of dimension $r-4$, for $i=1,2,3$. Moreover specialize the nine nodes on it, three on each $L'_i$, and consider the following exact sequence
\begin{eqnarray}\label{seconda succ transv}
0\rightarrow\LL_{r,2}(L_1,L_2,L_3)\rightarrow\mathscr{K}_2(r)\rightarrow\mathscr{K}_2(r-1).
\end{eqnarray}
The kernel system is empty, indeed in $\p^r$ there are no quadric hypersurfaces vanishing along three general subspaces of codimension three. 
\end{proof}

\begin{prop}
\label{kernel matching}
Let $L_1,L_2\cong\p^{r-3}$ be general subspaces of $\p^r$. The linear  system $\mathscr{K}_1(r):=\LL_{r,3}(\{L_1,2^{r-2}\},\{L_2,2^{r-2}\},2^3)$  is empty for $r\geq3$, $r\neq4$.
\end{prop}
\begin{proof}
We will prove the statement by induction on $r$, from $r-3$ to $r$, starting from the cases $r=3,5,7$.
For $r=3$, one has $\mathscr{K}_1(3)=\LL_{3,3}(2^5)$; this system is empty. 
For $r=5$ and $r=7$ it is an explicit computation. 
For $r=6$, $r\geq8$, we prove the statement exploiting the following restriction exact sequence:
$$
0\rightarrow \mathscr{K}_2(r)\rightarrow\mathscr{K}_1(r)\rightarrow\mathscr{K}_1(r-3).
$$
The kernel is empty by Proposition \ref{prop K_2}  and $\mathscr{K}_1(r-3)$ is empty by induction.

\end{proof}

\begin{prop}\label{kernel systems} 
Keep the same notation as above. Let $r\geq 5$. The matching linear system $\LL^m_V$ is empty.
\end{prop}
\begin{proof}
For $r=5$,  the matching system  is $\LL_{V}^m=\LL_{5,3}(\{L,2^3\},2^6,1)$. With an explicit computation, one can  see that it is empty. 
For $r=6$ the matching system is $\LL_V^m=\LL_{6,3}(\{L,2^5\},2^7)$. Restricting it to a general  $L_1\cong\p^3$ intersecting $L$ in the support $p$ of one of the nodes and specializing on it four general nodes, we get
$$
0\rightarrow\LL_{6,2}(\{L,2^4\},\{L_1,2^4\},2^3)\rightarrow \LL_{6,3}(\{L,2^5\},2^7)\rightarrow\LL_{3,3}(\{p,2\},2^4).
$$
The kernel is empty by Proposition \ref{kernel matching}, and the restricted system is $\LL_{3,3}(2^5)$ which is empty.
For $r=7$ the matching system is $\LL^m_V=\LL_{7,3}(\{L,2^7\},2^8)$, with $L\cong\p^4$. 
Let $p_1,\dots,p_7\in L$ and $q_1,\dots,q_8\in\p^r\setminus L$ be the supports of the fifteen nodes. Let $\pi$ be a hyperplane  such that $L\cap\pi=L'\cong\p^3$  and such that $p_4,\dots,p_7 \in L'$; moreover specialize on $\pi$ the points $q_2,\dots,q_8$:
$$
\LL^m_V  \stackrel{\phi}{\rightarrow}{\LL^m_V}_{|\pi}\subseteq \LL_{6,3}(\{L',2^4\},2^7).
$$
Notice that the line joining $q_1$ and $p_i$ is contained in all the sections of $\LL^m_V$ and it intersects $\pi$ at a point, for $i=1,2,3$. Therefore $
{\LL^m_V}_{|\pi}\subseteq\LL_{6,3}(\{L',2^4\},2^7,1^3).$
It is non-special as consequence of the previous point (case $r=6$), moreover it is empty, having virtual dimension equal to $-1$.  Furthermore the kernel of the restriction map $\phi$ is
$\LL_{7,2}(\{L,2^3\},2,1^7).$
It is easy to check that $\dim(\LL_{7,2}(\{L,2^3\},2))=6$, choosing for example $L=\{x_0=x_1=x_2=0\}$,
$p_1=[0,0,0,0,0,0,0,1]$, $p_2=[0,0,0,0,0,0,1,0]$,
$p_3=[0,0,0,0,0,1,0,0]$ and $q_1=[1,0,0,0,0,0,0,0]$.
Therefore, imposing seven further general base points, the resulting system is empty.
For $r\geq8$, the statement follows by induction restricting to a general $\p^{r-3}$ and making specializations of the points as follows:
$$
 \LL_{r,3}(\{L,2^{n^-(r-3,3)}\},2^{r+1},1^{\gamma(r)-\gamma(r-3)}) \rightarrow\LL_{r-3,3}(\{L',2^{n^-(r-6,3)}\},2^{r-2},1^{\gamma(r)-\gamma(r-3)}),
$$
where $L'=L\cap \p^{r-3}\cong\p^{r-6}$: the kernel is $\mathscr{K}_1(r)=\emptyset$.
\end{proof}

Finally, being $\LL^m_V=\hat{\LL}^m_T=\emptyset$, for $r\geq 5$, $r\neq 7$, then $\LL=\LL_t$ is empty.

\begin{rmk}
For $r=7$, the emptyness of the matching system does not suffice to conclude that the system of cubics of $\p^7$ with fifteen nodes is empty (and in particular non-special), because the kernel system $\hat{\LL}_T$ on the other component is not empty. Nevertheless this is crucial because it represents the starting point of the induction from $r-3$ to $r$, for $r\geq 10$, $r\equiv1\pmod3$; so we will analyse this case separately.
\end{rmk}

\subsubsection{Cubics in $\p^7$}

We will reproduce the same argument, but blowing up a subspace $L_1$ of codimension four, instead of three, in the central fiber of the trivial family $\p^7\times\Delta$. Let us denote by $T$ the exceptional component of the new special fiber, and with $V$ the strict transform, as above. Let $p:T\rightarrow L_1$ the natural map and $\pi:=p^\ast\mathcal{O}_{L_1}(1)$. Twist by $-T$ and consider the limit of the linear system of cubics of $\p^7$:
$$|3\pi+Q|  \textrm{ on } T \textrm{ and }
 |3H_V-Q| \textrm{ on } V,
$$
where $H_V$ is the pull-back of a hyperplane, and $Q$ is the exceptional divisor of the blow up.
Consider the system $\LL_{7,3}(2^{15})$ on the general fiber. To prove emptyness, we use the same trick as in the general case: we specialize the right number of points on the two components as follows: 
$$
\LL_T=|3\pi+Q-2^5|,\textrm{ and }
\LL_V=|3H_V-Q-2^{10}|\cong\LL_{7,3}(L_1,2^{10})
$$
The kernels of the restriction to $Q$ are
$$
\hat{\LL}_T\cong\LL_{3,3}(2^{5})=\emptyset,\textrm{ and }
\hat{\LL}_V=|3H_V-2Q-2^{10}|\cong\LL_{7,3}(2L_1,2^{10}).
$$

Each node specialized on $T$ selects a fiber of the ruling of $T$. Each fiber cuts a $\p^3$ at the intersection $Q$, which corresponds to a fiber of the ruling of $Q$. So, as in the general case, the matching system on $V$ is
$$
\LL_{V,7}^m\cong\LL_{7,3}(\{L_1,2^5\},2^{10})
$$
and it has virtual dimension $-1$. To prove its emptyness, we make two subsequent specialization of the general nodes, five on  $L_2$ and five on $L_3$, where $L_2,L_3\cong\p^3$ are general subspaces of $\p^7$, as we did for the general case:
$$
0\rightarrow  \mathcal{K}_1\rightarrow \LL_{V,7}^m \rightarrow \LL_{3,3}(2^5)\rightarrow 0
$$
where $\mathcal{K}_1:=\LL_{7,3}(\{L_1,2^5\},\{L_2,2^5\},2^5)$ and 
$$
0\rightarrow  \mathcal{K}_2\rightarrow \mathcal{K}_1\rightarrow \LL_{3,3}(2^5)\rightarrow 0
$$
where $\mathcal{K}_2:=\LL_{7,3}(\{L_1,2^5\},\{L_2,2^5\},\{L_3,2^5\})$. 
With an explicit computation we prove that $\mathcal{K}_2$ is empty, therefore $\mathcal{K}_1$ is empty too and this conludes the proof.

\section*{Acknowledgements} This paper is part of my Ph.D. thesis. I am extremely grateful to Professor Ciro Ciliberto who has supported me thoughout this work with his patience and very deep knowledge.


\end{document}